\documentclass[12pt,a4paper,reqno]{amsart}

\usepackage[colorlinks, linkcolor=red,anchorcolor=blue,citecolor=green]{hyperref}
\usepackage{graphics,epic}
\usepackage{amsmath,amssymb, amsthm}
\usepackage[all,2cell]{xy}
\usepackage{times}
\usepackage{mathrsfs}

\newtheorem{theorem}{Theorem}[section]
\newtheorem*{theorem*}{Theorem}

\newtheorem{lemma}[theorem]{Lemma}
\newtheorem{proposition}[theorem]{Proposition}
\newtheorem{corollary}[theorem]{Corollary}

\newtheorem*{conjecture*}{Conjecture}

\theoremstyle{definition}

\newtheorem{example}[theorem]{Example}
\newtheorem{remark}[theorem]{Remark}

\newcommand{\ie}{{\em i.e.}\ }

\newcommand{\opname}[1]{\operatorname{\mathsf{#1}}}

\renewcommand{\mod}{\opname{mod}\nolimits}

\newcommand{\add}{\opname{add}\nolimits}

\newcommand{\der}{\cd}

\newcommand{\rank}{\opname{rank}\nolimits}

\newcommand{\ind}{\opname{ind}}

\newcommand{\Sub}{\opname{Sub}}
\newcommand{\Fac}{\opname{Fac}}

\newcommand{\pd}{\opname{pd}}
\newcommand{\id}{\opname{id}}
\newcommand{\gld}{\opname{gl.dim}}
\newcommand{\Du}{\opname{D}\nolimits}
\newcommand{\thick}{\opname{thick}}

\newcommand{\im}{\opname{im}\nolimits}

\renewcommand{\P}{\mathbb{P}}

\newcommand{\Hom}{\opname{Hom}}

\newcommand{\Ext}{\opname{Ext}}

\newcommand{\End}{\opname{End}}

\newcommand{\cd}{{\mathcal D}}

\newcommand{\n}{\mathfrak{N}}

\newcommand{\xra}{\xrightarrow}

\newcommand{\bsm}{\begin{smallmatrix}}
\newcommand{\esm}{\end{smallmatrix}}

\newcommand{\bbsm}{\left[\begin{smallmatrix}}
\newcommand{\besm}{\end{smallmatrix}\right]}

\newcommand{\bbm}{\begin{matrix}}
\newcommand{\ebm}{\end{matrix}}

\begin{document}

\title[On endomorphism algebras of silting complexes]{On endomorphism algebras of silting complexes over hereditary abelian categories}\thanks{This work was supported by the National Natural Science Foundation of China (Grant No. 12571040)}

\author[Dai]{Wei Dai}
\address{Wei Dai\\ Department of Mathematics\\SiChuan University\\610064 Chengdu\\P.R.China}
\email{daiweimath@gmail.com}

\author[Fu]{Changjian Fu}
\address{Changjian Fu\\Department of Mathematics\\SiChuan University\\610064 Chengdu\\P.R.China}
\email{changjianfu@scu.edu.cn}

\author[Peng]{Liangang Peng}
\address{Liangang Peng\\Department of Mathematics\\SiChuan University\\610064 Chengdu\\P.R.China}
\email{penglg@scu.edu.cn}


\begin{abstract}
Let $\mathcal{E}$ be the class of finite-dimensional algebras isomorphic to endomorphism algebras of silting complexes over hereditary abelian categories. It is proved that the class $\mathcal{E}$ is closed under taking  idempotent subalgebras, idempotent quotients and $\tau$-tilting reduction. We also show that the proper class consisting of quasi-silted algebras is also closed under these operations. In addition, several classical classes of algebras---including laura, glued, and weakly shod algebras---are  proved to be closed under idempotent quotients, thereby generalizing a result originally established for specific idempotents.
\end{abstract}

\maketitle


\section{Introduction}

For a given algebra $A$, the study of its idempotent subalgebra $eAe$ and the idempotent quotient $A/AeA$ associated with an idempotent $e \in A$ is a central theme in representation theory (see, for instance, \cite{CPS88, AC04, X, Xu, CK25}). A primary motivation for focusing on these two constructions originates from the theory of recollements of derived categories. It is well known that under suitable conditions, the bounded derived category $\mathcal{D}^b(A)$ can be glued from $\mathcal{D}^b(eAe)$ and $\mathcal{D}^b(A/AeA)$ via a recollement \cite{CPS88}. This framework provides a powerful mechanism for homological reduction, allowing one to decompose the homological properties of $A$---such as the finiteness of its global dimension---into those of $eAe$ and $A/AeA$ \cite{AKLY}.

From this perspective, establishing the closure properties of a specific class of algebras under these operations is of fundamental importance. It transforms the class into a self-contained system where global conjectures and structural problems can be tackled through recursive arguments. For example, the fact that the class of $2$-Calabi-Yau tilted algebras arising from hereditary categories is closed under idempotent quotients was a key step in proving the connectedness of cluster-tilting graphs for hereditary categories \cite{FG21}. Similarly, the proof of the connectedness of $\tau$-tilting graphs for gentle algebras  relies on the fact that the class of gentle algebras is closed under $\tau$-tilting reduction \cite{FGLZ}.

In the representation theory of finite-dimensional algebras, the study of endomorphism algebras of distinguished objects---such as tilting modules, support $\tau$-tilting modules, and (2-term) silting complexes---lies at the heart of the field. These constructions give rise to several fundamental classes of algebras, including tilted, quasi-tilted, silted, and quasi-silted algebras. A landmark result by Happel, Reiten, and Smal{\o} \cite{HRS} established that quasi-tilted algebras are precisely the endomorphism algebras of tilting objects in hereditary abelian categories. This categorical characterization was further extended by Buan and Zhou \cite{BZ16a, BZ16b}, who proved that shod algebras can be realized as endomorphism algebras of 2-term silting complexes over hereditary abelian categories. As a further generalization, endomorphism algebras of silting complexes over finite-dimensional hereditary algebras were investigated in \cite{ALLT} from the viewpoint of categorical structure.

The purpose of this work is to investigate the class $\mathcal{E}$ of finite-dimensional algebras realized as endomorphism algebras of basic silting complexes over hereditary abelian categories, specifically focusing on their behavior under the operations of taking idempotent subalgebras and idempotent quotients. Employing the machinery of silting reduction in triangulated categories, we develop an approach to analyze the stability of this class of algebras. Our main results, which are interconnected and established via this method, verify the following closure properties.

\begin{theorem}[Theorem \ref{p: idempotent-factor-closed}, Theorem \ref{p: idempotent-subalg-closed}] \label{t: introduction_1}
Let $A\in \mathcal{E}$ and $e\in A$ be an idempotent. Then
\begin{itemize}
    \item[(1)] $A/AeA\in \mathcal{E}$. Moreover, if $A$ is quasi-silted, then so is $A/AeA$.
    \item[(2)] $eAe\in \mathcal{E}$. Moreover, if $A$ is quasi-silted (resp. quasi-tilted), then so is $eAe$.
\end{itemize}
\end{theorem}
We remark that the final assertion of (2) was previously established in \cite{AC04}. However, our proof is entirely different from the one provided there.

Note that the idempotent quotient can be viewed as a special case of the $\tau$-tilting reduction introduced by Jasso \cite{J}. It is therefore natural to ask whether the class $\mathcal{E}$ is closed under the more general operation of $\tau$-tilting reduction.

\begin{theorem} [Theorem \ref{t: quasisilted is tau-tilting reduction closed}] \label{t: introduction_2}
Let $A\in \mathcal{E}$ and  $Z$ be a $\tau$-rigid $A$-module.
\begin{itemize}
    \item[(1)] The $\tau$-tilting reduction of $A$ with respect to $Z$ belongs to $\mathcal{E}$.
    \item[(2)] If $A$ is quasi-silted, then its $\tau$-tilting reduction with respect to $Z$ is also quasi-silted.
\end{itemize}
\end{theorem}

The proof of Theorem \ref{t: introduction_1}, which establishes the closure property under idempotent quotients, motivates a further exploration into the classical classes of algebras studied by Assem and Coelho \cite{AC04}. Specifically, we examine whether these well-known classes---namely laura, glued, weakly shod, and shod algebras---are also closed under quotients by arbitrary idempotents. By employing an independent analysis with techniques distinct from silting reduction, we provide an affirmative answer to this question.

\begin{theorem} [Theorem \ref{t: quotient-closed}] \label{t: introduction_3}
Let $A$ be a finite-dimensional $k$-algebra and $e\in A$ be an idempotent.
\begin{enumerate}
\item If $A$ is a laura algebra, then so is $A / AeA$.
\item If $A$ is a right (or left) glued algebra, then so is $A / AeA$.
\item If $A$ is a weakly shod algebra, then so is $A / AeA$.
\item If $A$ is a shod algebra, then so is $A / AeA$.
\end{enumerate}
\end{theorem}

This result significantly generalizes the work of Zito \cite{Z}, who proved a similar statement for specific idempotents.

The paper is organized as follows. We provide the necessary preliminaries and the main method of silting reduction in Section \ref{s:preliminary}. In Section \ref{s:closure properties}, we present the proofs of Theorems \ref{t: introduction_1} and \ref{t: introduction_2}. Section \ref{s: Quotient-closed property by idempotents} is devoted to the proof of Theorem \ref{t: introduction_3}. Finally, Section \ref{s:examples} provides several examples to illustrate the boundaries of our results.

\subsection*{Notation}
Throughout this paper, let $k$ be a field. By a hereditary abelian category, we mean a $\Hom$-finite and $\Ext$-finite $k$-linear hereditary abelian category.
\begin{itemize}
\item Let $\mathcal{X}$ be a subcategory of a category $\mathcal{C}$. We denote by $\mathcal{X}^\perp$ the full subcategory of $\mathcal{C}$ consisting of objects $X \in \mathcal{C}$ such that $\Hom_\mathcal{C}(\mathcal{X}, X) = 0$. The subcategory $^\perp\mathcal{X}$ is defined dually.

\item We denote by $[1]$ the suspension functor of a triangulated category $\mathcal{T}$. For subcategories $\mathcal{X}$ and $\mathcal{Y}$ of $\mathcal{T}$, let $\mathcal{X} * \mathcal{Y}$ be the full subcategory consisting of objects $E \in \mathcal{T}$ such that there is a triangle $X \to E \to Y \to X[1]$ with $X \in \mathcal{X}$ and $Y \in \mathcal{Y}$. Given an object $X\in\mathcal{T}$, we denote by $\thick X$ the smallest thick subcategory of $\mathcal{T}$ containing $X$.
\item In a Krull-Schmidt category, for an object $X$, denote by $|X|$ the number of pairwise non-isomorphic indecomposable direct summands of $X$, and by $\add X$ the full subcategory consisting of finite direct sums of direct summands of $X$.
\item For an algebra $A$, we denote by $\mod A$ the category of finitely generated right $A$-modules, and by $\ind A$ the set of isomorphism classes of indecomposable right $A$-modules. For an element $x \in A$, let $\langle x \rangle = AxA$ be the two-sided ideal generated by $x$.
\item For a module $M \in \mod A$, we denote by $\Fac M$ the full subcategory of factor modules of direct sum of finite copies of $M$. Dually, denote by $\Sub M$ the full subcategory of the submodules of direct sum of finite copies of $M$.

\end{itemize}

\section{Preliminaries}  \label{s:preliminary}

\subsection{Perpendicular subcategories in triangulated categories} 
Let $\mathcal{T}$ be a triangulated category and  $\mathcal{X}$  a subcategory of $\mathcal{T}$. A morphism $f: X \to Y$ is called a {\em right $\mathcal{X}$-approximation} of $Y \in \mathcal{T}$ if $X \in \mathcal{X}$ and the induced map $\Hom_\mathcal{T}(\mathcal{X}, f): \Hom_\mathcal{T}(\mathcal{X}, X) \to \Hom_\mathcal{T}(\mathcal{X}, Y)$ is surjective. The subcategory $\mathcal{X}$ is {\em contravariantly finite} in $\mathcal{T}$ if every object in $\mathcal{T}$ admits a right $\mathcal{X}$-approximation. Dually, one defines {\em left $\mathcal{X}$-approximation} and {\em covariantly finite subcategories}. The subcategory $\mathcal{X}$ is called {\em functorially finite} if it is both contravariantly finite and covariantly finite. 

\begin{lemma}\textnormal{\cite[Proposition 5.33]{IO}}\label{l:X*Y-functorially-finite}
Let  $\mathcal{X}$ and $\mathcal{Y}$ be full subcategories of $\mathcal{T}$. 
\begin{itemize}
    \item[(1)] If $\mathcal{X}$ and $\mathcal{Y}$ are contravariantly finite in $\mathcal{T}$, then so is $\mathcal{X} \ast \mathcal{Y}$.
    \item[(2)] If $\mathcal{X}$ and $\mathcal{Y}$ are covariantly finite in $\mathcal{T}$, then so is $\mathcal{X} * \mathcal{Y}$.
\end{itemize}
\end{lemma}

Recall that a pair of subcategories $(\mathcal{X}, \mathcal{Y})$ of $\mathcal{T}$ is a {\em torsion pair} of $\mathcal{T}$, if  $\Hom_{\mathcal{T}}(\mathcal{X}, \mathcal{Y})=0$ and $\mathcal{X}*\mathcal{Y}=\mathcal{T}$. The following useful result is known as Wakamatsu's Lemma (see \cite[Lemma 2.22]{AiharaIyama}).
\begin{lemma}\label{l:wakamatsu-lemma}
Let $\mathcal{M}$ be a subcategory of $\mathcal{T}$ such that $\mathcal{M}*\mathcal{M}\subseteq \mathcal{M}$.
\begin{itemize}
\item[(1)]  If $\mathcal{M}$ is contravariantly finite, then $(\mathcal{M}, \mathcal{M}^\perp)$ is a torsion pair of $\mathcal{T}$;
\item[(2)] If $\mathcal{M}$ is covariantly finite, then $(\!^\perp\mathcal{M}, \mathcal{M})$ is a torsion pair of $\mathcal{T}$.
\end{itemize}
\end{lemma}

Let $\mathcal{S}$ be a thick subcategory  of $\mathcal{T}$, denote by $\mathcal{T}/\mathcal{S}$ the Verdier quotient of $\mathcal{T}$ with respect to $\mathcal{S}$. Denote by $\mathbb{L}: \mathcal{T}\to \mathcal{T}/\mathcal{S}$ the localization functor. The following  is also well-known, see \cite[Lemma 2.2]{DF} for instance.

\begin{lemma}\label{l:perp-quotient}
Let $\mathcal{S}$ be a thick subcategory of $\mathcal{T}$.
\begin{enumerate}
\item If $\mathcal{S}$ is contravariantly finite, then the restriction $\mathcal{S}^\perp\to \mathcal{T}/\mathcal{S}$ of $\mathbb{L}$ is a triangle equivalence.
\item If  $\mathcal{S}$ is covariantly finite, then the restriction $^\perp\mathcal{S}\to \mathcal{T}/\mathcal{S}$ of $\mathbb{L}$ is a triangle equivalence.
\end{enumerate}
\end{lemma}

\subsection{Silting theory} Let $\mathcal{T}$ be a $\Hom$-finite Krull-Schmidt triangulated category. 
An object $T\in \mathcal{T}$ is  called  {\em presilting}  if $\Hom_{\mathcal{T}}(T,T[>0])=0$.
A presilting object $T\in \mathcal{T}$ is said to be {\em silting} if $\thick T=\mathcal{T}$. A silting object $T$ is called \emph{tilting} if it also satisfies $\Hom_{\mathcal{T}}(T,T[<0])=0$. The following result provides a characterization of the thick subcategory generated by a presilting object.
\begin{lemma}\cite[Proposition 2.17]{AiharaIyama}\label{lem:thick-subcat-as-extension}
    Let $D\in \mathcal{T}$ be a presilting object. Then 
    \[
    \thick D=\bigcup\limits_{l \geq 0} \add (D[-l]) * \add (D[-l + 1]) * \cdots * \add (D[l]).
    \]
\end{lemma}
While there exist triangulated categories that do not possess silting objects, if $\mathcal{T}$ admits a silting object $T$, then its Grothendieck group $K_0(\mathcal{T})$ is a free abelian group of rank $|T|$ \cite[Theorem 2.27]{AiharaIyama}. Consequently, all silting objects in $\mathcal{T}$ have the same number of pairwise non-isomorphic indecomposable direct summands.

Let $T=M\oplus \overline{T}$ be a basic silting object of $\mathcal{T}$. Consider the triangle
\[M\xrightarrow{f_M} T_M\to N\to M[1],
\]
where $f_M$ is a minimal left $\add \overline{T}$-approximation of $M$. According to \cite[Theorem 2.31]{AiharaIyama},  the object $\mu_M^{-}(T) := N \oplus \overline{T}$ is again a basic silting object of $\mathcal{T}$, called the {\em left mutation} of $T$ with respect to $M$.

We denote by $\opname{silt} \mathcal{T}$ the set of isomorphism classes of basic silting objects in $\mathcal{T}$.

\begin{theorem}\cite[Theorem 2.37]{AiharaIyama}\label{t:silting-reduction}
	Let $\mathcal{T}$ be a Krull-Schmidt triangulated category, $\mathcal{S}$ a functorially finite thick subcategory of $\mathcal{T}$ and $\mathcal{T}/\mathcal{S}$ the Verdier quotient. Denote by $\mathbb{L}:\mathcal{T}\to \mathcal{T}/\mathcal{S}$ the localization functor. For any $D\in \opname{silt}\ \mathcal{S}$,  there is a bijective map
	\[\{T\in \opname{silt}\ \mathcal{T}\mid D\in \add T\}\to \opname{silt}\ \mathcal{T}/\mathcal{S}
	\] given by $T\mapsto \mathbb{L}(T)$.
\end{theorem}
Let us recall the inverse map of the aforementioned bijection, following the proof of \cite[Theorem 2.37]{AiharaIyama}.
Let $\mathcal{S}_D^{\leq 0}:=\cup_{l\geq 0}\add D*\add D[1]*\cdots*\add D[l]$ and $\mathcal{S}_{D}^{<0}:=\mathcal{S}_D^{\leq 0}[1]$. It is known that $\mathcal{S}_{D}^{<0}$ is covariantly finite in $\mathcal{T}$.  Furthermore, as $\mathcal{S}$ is functorially finite, we can identify $\mathcal{T}/\mathcal{S}$ with $\mathcal{S}^\perp$. Let $N\in \mathcal{S}^\perp$ be a silting object in $\mathcal{S}^\perp$.
Consider the following triangle:
\[
S_N\to T_N\to N\xrightarrow{g}S_N[1],
\]
where $S_N[1]\in \mathcal{S}_{D}^{<0}$ and $g$ is a minimal left $\mathcal{S}_{D}^{<0}$-approximation of $N$. According to the proof of \cite[Theorem 2.37]{AiharaIyama},  $T_N\oplus D$ is a silting object in $\mathcal{T}$ such that $\mathbb{L}(T_N\oplus D)=\mathbb{L}(N)$.

\subsection{Perpendicular subcategories in hereditary categories} Let $\mathcal{H}$ be a hereditary abelian category, i.e., an abelian category satisfying $\Ext_{\mathcal{H}} ^2 (-, -) = 0$. For any object $E \in \mathcal{H}$, we define
\[
E^{\perp_{[0, 1]}} = \{X \in \mathcal{H}\mid \Hom_{\mathcal{H}}(E, X) = 0 = \Ext_{\mathcal{H}}^1 (E, X) \}.
\]
Then we have
\begin{lemma}\label{l: perp-cat-is-hereditary}
The subcategory $E^{\perp_{[0, 1]}}$ is a hereditary abelian subcategory of $\mathcal{H}$. In particular $E^{\perp_{[0, 1]}}$ is closed under taking images, kernels, and cokernels.
\end{lemma}
\begin{proof}
One can readily verify that $E^{\perp_{[0, 1]}}$ is closed under taking images, kernels, and cokernels in $\mathcal{H}$; consequently, it is an abelian subcategory of $\mathcal{H}$. It suffices to show that $\Ext_{E^{\perp_{[0, 1]}}} ^2 (X, Y) = 0$ for any $X, Y \in E^{\perp_{[0, 1]}}$. Let $\varepsilon \in \Ext_{E^{\perp_{[0, 1]}}}^2(X, Y)$ be a $2$-extension represented by the following exact sequence in $E^{\perp_{[0, 1]}}$:
\[
\xymatrix@-1.8pc@R=10pt{
0  \ar[rr] &&Y \ar[rr] && E_1 \ar[rr]^f \ar[dr]_u &&E_2 \ar[rr] &&X \ar[rr] &&0. \\
&&&&& \im f \ar[ur]_v
}
\]
Since $E^{\perp_{[0, 1]}}$ is an abelian subcategory of $\mathcal{H}$, we obtain the following two exact sequences in $E^{\perp_{[0, 1]}}$:
\begin{align}
0 \to Y \to E_1 \xra{u} \im f \to 0,\label{exseq:21}\\
0 \to \im f \xra{v} E_2 \to X \to 0.
\end{align}
Let these sequences represent $\varepsilon_1 \in \Ext_{E^{\perp_{[0, 1]}}}^1(\im f, Y)$ and $\varepsilon_2 \in \Ext_{E^{\perp_{[0, 1]}}}^1(X, \im f)$, respectively.
 Recall that for any $M, N \in E^{\perp_{[0, 1]}}$, $\Ext_{E^{\perp_{[0, 1]}}}^1 (M, N) = \Ext_{\mathcal{H}}^1 (M, N)$. Applying the functor $\Hom_{E^{\perp_{[0, 1]}}}(X, -)$ to \eqref{exseq:21} yields the following exact sequence
\[
\cdots \to \Ext_{\mathcal{H}}^1 (X, E_1) \xra{u_*} \Ext_{\mathcal{H}}^1 (X, \im f) \xra{\eta} \Ext_{E^{\perp_{[0, 1]}}}^2 (X, Y) \to \cdots.
\]
By definition,  $\varepsilon$ is the image of $\varepsilon_2$ under the connecting map $\eta$, i.e., $\varepsilon = \eta(\varepsilon_2)$. On the other hand, since $\mathcal{H}$ is hereditary, the map $u_*$ is surjective, which implies $\eta = 0$. It follows that  $\varepsilon = 0$. As $\varepsilon$ was chosen arbitrarily, we conclude that $\Ext_{E^{\perp_{[0, 1]}}} ^2 (X, Y) = 0$.
\end{proof}

Let  $\der^b (\mathcal{H})$ be the bounded derived category of $\mathcal{H}$. Generalizing the definition of $E^{\perp_{[0, 1]}}$ to an arbitrary object $D \in \mathcal{D}^b(\mathcal{H})$, we consider the subcategory
\[
D^{\perp_{\mathbb{Z}}} = \{M \in \mathcal{H}\mid \Hom_{\der^b(\mathcal{H})}(D, M[\mathbb{Z}]) = 0 \}.
\]
The following result extends Lemma \ref{l: perp-cat-is-hereditary} to this setting.

\begin{proposition}\label{p: perp-for-presilting}
The subcategory $D^{\perp_{\mathbb{Z}}}$ is a hereditary abelian subcategory of $\mathcal{H}$. In particular, $D^{\perp_{\mathbb{Z}}}$ is closed under taking images, kernels, and cokernels.
\end{proposition}
\begin{proof}
Let $D = D_1 \oplus D_2 \oplus \cdots \oplus D_r$ be the decomposition of $D$ into indecomposable objects. Then, each $D_i$ is of the form $X_i[n_i]$ for some indecomposable object $X_i \in \mathcal{H}$ and $n_i \in \mathbb{Z}$. By the definition of $D^{\perp_{\mathbb{Z}}}$, we have
\[D^{\perp_{\mathbb{Z}}} = \bigcap D_i^{\perp_{\mathbb{Z}}} = \bigcap X_i^{\perp_{\mathbb{Z}}} = \bigcap X_i^{\perp_{[0, 1]}}=(\oplus X_i)^{\perp_{[0,1]}}.\]  The assertion then follows immediately from Lemma \ref{l: perp-cat-is-hereditary}.
\end{proof}

\begin{proposition}\label{p:identification-subcategory}
    Let $D\in \der^b(\mathcal{H})$. Then there exists a triangle equivalence
\[\iota : \der^b(D^{\perp_{\mathbb{Z}}}) \xrightarrow{\sim} (\thick{D})^\perp.\]
\end{proposition}
\begin{proof}
    Consider the inclusion $D^{\perp_{\mathbb{Z}}} \hookrightarrow \mathcal{H}$. In view of Proposition \ref{p: perp-for-presilting}, this inclusion extends to a triangle functor $\iota : \der^b(D^{\perp_{\mathbb{Z}}}) \to \der^b(\mathcal{H})$, which maps a complex to itself. For any objects $X, Y \in D^{\perp_{\mathbb{Z}}}$, it follows from Proposition \ref{p: perp-for-presilting} that
\begin{align*}
&\Hom_{\der^b(D^{\perp_{\mathbb{Z}}})} (X, Y) = \Hom_{D^{\perp_{\mathbb{Z}}}} (X, Y) = \Hom_{\mathcal{H}} (X, Y) = \Hom_{\der^b(\mathcal{H})} (X, Y),\\
&\Hom_{\der^b(D^{\perp_{\mathbb{Z}}})} (X, Y[1]) = \Ext_{D^{\perp_{\mathbb{Z}}}}^1 (X, Y) = \Ext_{\mathcal{H}}^1 (X, Y) = \Hom_{\der^b(\mathcal{H})} (X, Y[1]).
\end{align*}
This implies that the induced triangle functor $\iota$ is fully faithful. Thus, we may view $\der^b(D^{\perp_{\mathbb{Z}}})$ as a full subcategory of $\der^b(\mathcal{H})$.

Let $\mathcal{S} = \thick{D}$. Clearly, any complex in $\der^b(D^{\perp_{\mathbb{Z}}})$ belongs to $\mathcal{S}^{\perp}$. Conversely, for any $X \in \mathcal{S}^{\perp}$, the fact that $\mathcal{H}$ is hereditary implies $X \cong \bigoplus_i H^i(X)[-i]$, which lies in $\der^b(D^{\perp_{\mathbb{Z}}})$.  Consequently, the image of $\iota$ is exactly $\mathcal{S}^\perp$, and thus $\iota : \der^b(D^{\perp_{\mathbb{Z}}}) \to (\thick{D})^\perp$ is a triangle equivalence.
\end{proof}

\subsection{Rigid objects and presilting complexes over hereditary categories}
Let $\mathcal{H}$ be a hereditary abelian category. Recall that an object $E \in \mathcal{H}$ is called \emph{rigid} if $\Ext_{\mathcal{H}}^1 (E, E) = 0$. In this context,
a presilting (resp. silting, tilting) object $T\in \der^b(\mathcal{H})$ is also referred to as a {\em presilting (resp. silting, tilting) complex} over $\mathcal{H}$.

The following fundamental result is due to Happel and Ringel~\cite{HappelRingel}.
\begin{lemma}\label{l:HR-exceptional}
Let $E$ and $F$ be indecomposable objects in $\mathcal{H}$ such that $\Hom_{\der^b(\mathcal{H})}(F, E[1])=0$. Then any nonzero homomorphism $f:E\to F$ is either a monomorphism or an epimorphism. In particular, the endomorphism ring of any indecomposable rigid object is a division algebra.
\end{lemma}

Let $\mathcal{X}$ be a full subcategory of $\der^b(\mathcal{H})$ and $X\in \mathcal{X}$ an indecomposable object. A path in $\mathcal{X}$ from $X$ to itself is called a {\em cycle} in $\mathcal{X}$, that is, a sequence of nonzero non-isomorphisms between indecomposable objects in $\mathcal{X}$ of the form
\[X=X_0\xrightarrow{f_1}X_1\xrightarrow{f_2}X_2\to \cdots\xrightarrow{f_r}X_r=X.
\]
The following result is a consequence of Lemma~\ref{l:HR-exceptional} (see, e.g., \cite[Lemma 4.2]{Fu} or \cite[Corollary 4.2]{HappelRingel}).
\begin{lemma}\label{l:no-cycle}
Let $T$ be an object in $\der^b(\mathcal{H})$ such that $\Hom_{\der^b(\mathcal{H})}(T, T[1])=0$. Then the subcategory $\add T$ has no cycle.
\end{lemma}

Lemma \ref{l:no-cycle} yields the following consequence.

\begin{lemma}\label{l:no-cycle-plus} Let $T$ be an object in $\der^b(\mathcal{H})$ such that $\Hom_{\der^b(\mathcal{H})}(T, T[1])=0$.
\begin{itemize}
    \item[(a)]  There is a unique decomposition $T =  T_1 \oplus T_2 \oplus \cdots \oplus T_r$ into indecomposable summands $T_i \in \mathcal{H}[n_i]$ with integers $n_1 \leq n_2 \leq \cdots \leq n_r$, such that $\Hom_{\der^b(\mathcal{H})}(T_j, T_i) = 0$ for all $i < j$.
    \item[(b)] Furthermore, if $T$ is a presilting object, then 
    \[
\thick T = \thick T_r * \thick T_{r-1} * \cdots * \thick T_1.
\]
\end{itemize}
\end{lemma}

\begin{proof}
Since $\mathcal{H}$ is hereditary, we can assume $T = X_1[m_1] \oplus X_2[m_2] \oplus \cdots \oplus X_t[m_t]$, where each $X_k \in \mathcal{H}$ and $m_1 < m_2 < \cdots < m_t$. The condition $\Hom_{\der^b(\mathcal{H})}(T,T[1])=0$ implies that each $X_k$ is a rigid object in $\mathcal{H}$. Applying Lemma \ref{l:no-cycle} to each $X_k$,  we may further write $X_k = X_{k1} \oplus X_{k2} \oplus \cdots \oplus X_{kr_k}$, where each $X_{ki}$ is indecomposable and $\Hom_{\mathcal{H}}(X_{kj}, X_{ki}) = 0$ for all $i < j$. Setting $n_{ki} := m_k$ and $T_{ki} := X_{ki}[n_{ki}] = X_{ki}[m_k]$ for $1 \leq k \leq t$ and $1 \leq i \leq r_k$, the resulting decomposition \[T = T_{11} \oplus \cdots \oplus T_{1r_1} \oplus \cdots \oplus T_{t1} \oplus \cdots \oplus T_{tr_t}\] is easily seen to satisfy the assertion in (a).

We now turn to the statement $(b)$. Let $T \cong T_1 \oplus T_2 \oplus \cdots \oplus T_r$ be the canonical decomposition obtained in $(a)$, where each $T_i$ is an indecomposable object in $\mathcal{H}[n_i]$ with $n_1 \leq n_2 \leq \cdots \leq n_r$, satisfying $\Hom_{\mathcal{D}^b(\mathcal{H})}(T_j, T_i) = 0$ for all $i < j$. 
 Since $T$ is a presilting object, the orthogonality condition extends to $\Hom_{\mathcal{D}^b(\mathcal{H})}(T_j, T_i[k]) = 0$ for all $k \in \mathbb{Z}$ and $i < j$. 
 Consequently, one has $\Hom_{\der^b(\mathcal{H})}(\thick T_j, \thick T_i) = 0$ whenever $i < j$. 
 According to \cite[Lemma 2.22 (a)]{AiharaIyama},  the subcategory $\mathcal{U} := \thick T_r * \thick T_{r-1} * \cdots * \thick T_1$ is closed under direct summands, and is therefore a thick subcategory of $\der^b(\mathcal{H})$. Since $T \in \mathcal{U}$ and $\mathcal{U} \subseteq \thick T$ by construction, we conclude that $\thick T = \mathcal{U}$.
\end{proof}

\begin{lemma}\label{l:functorially-finite}
Let $T$ be a presilting object in $\der^b (\mathcal{H})$, then $\thick T$ is functorially finite in $\der^b (\mathcal{H})$.
\end{lemma}
\begin{proof} 
We fix a decomposition of $T=T_1\oplus\cdots\oplus T_r$ as in Lemma \ref{l:no-cycle-plus} (a). Since each $T_i$ is indecomposable and rigid, it follows that the indecomposable objects in $\operatorname{thick} T_i$ are precisely the shifts $T_i[j]$ for $j \in \mathbb{Z}$. As a result, $\thick T_i$ is functorially finite for each $1\leq i\leq r$. Combining this with the decomposition $\operatorname{thick} T = \operatorname{thick} T_r * \cdots * \operatorname{thick} T_1$ and Lemma \ref{l:X*Y-functorially-finite}, we conclude that $\thick T$ is functorially finite.

\end{proof}

\begin{proposition}\label{p: quotient-by-presilting}
Let $D$ be a presilting object in $\der^b(\mathcal{H})$, then there is a triangle equivalence \[ \der^b(D^{\perp_{\mathbb{Z}}}) \cong \der^b(\mathcal{H})/\thick{D}. \]
\end{proposition}
\begin{proof}
By Lemma \ref{l:functorially-finite}, $\thick D$ is functorially finite in $\der^b(\mathcal{H})$. It follows that there is a triangle equivalence $\mathbb{L}|_{(\thick D)^\perp}: (\thick D)^\perp\to \der^b(\mathcal{H})/\thick D$  by Lemma \ref{l:perp-quotient}.  On the other hand, we have a triangle equivalence $\iota: \der^b(D^{\perp_\mathbb{Z}})\to (\thick D)^\perp$ by Proposition \ref{p:identification-subcategory}.
By composing these two equivalences, the assertion follows immediately.
\end{proof}

\subsection{Completion} We retain the notation from the previous subsection. The following result from \cite{DF} indicates that any presilitng object in $\der^b(\mathcal{H})$ can be completed to a silting complex, see also \cite{BY}.

\begin{theorem}\cite[Theorems 1.1, 1.2]{DF} \label{t:completion-for-presilting}
Suppose $\der^b(\mathcal{H})$ admits a silting object, then every presilting object of $\der^b(\mathcal{H})$ is a direct summand of a silting object.
\end{theorem}

From now on, we assume that $\der^b(\mathcal{H})$ admits a silting object.
Denote by $K_0(\mathcal{H})$ the Grothendieck group of $\der^b(\mathcal{H})$, which is a free abelian group of finite rank. Denote by $\rank K_0(\mathcal{H})$ the rank of $K_0(\mathcal{H})$. We have the following consequence.

\begin{corollary}\label{c:cardinality}
Let $D \in \der^b(\mathcal{H})$ be a presilting object. Then $D$ is silting if and only if $|D| = \rank K_0(\mathcal{H})$.
\end{corollary}
\begin{proof}
Assume that $D$ is basic, and let $D \oplus N$ be a basic silting object, whose existence is guaranteed by Theorem \ref{t:completion-for-presilting}. According to \cite[Theorem 2.27]{AiharaIyama}, $|T| = \rank K_0(\mathcal{H})$ for any silting object $T \in \der^b(\mathcal{H})$. Hence $|D| + |N| = \rank K_0(\mathcal{H})$. Then $|D| = \rank K_0(\mathcal{H}) \iff N = 0 \iff$ $D$ is silting.
\end{proof}

A presilting object $D$ is called \emph{almost silting} if $|D| = \rank K_0(\mathcal{H}) - 1$.
\begin{proposition}\label{p: ind-perp-completion}
Let $N$ be an almost silting object in $\der^b(\mathcal{H})$. Then there exists an indecomposable presilting object $D$, such that $D \oplus N$ is silting and $N \in (\thick D)^\perp$.
\end{proposition}
\begin{proof}
Let $\rank K_0(\mathcal{H})=n$.
Without loss of generality, assume that $N$ is basic. By Lemma \ref{l:no-cycle-plus} (a), we fix a decomposition  $N =  N_1 \oplus N_2 \oplus \cdots \oplus N_{n-1}$ of $N$, where each $N_i$ is an indecomposable summand such that $N_i \in \mathcal{H}[m_i]$ with $m_1 \leq m_2 \leq \cdots \leq m_{n-1}$ in $\mathbb{Z}$, and $\Hom_{\der^b(\mathcal{H})}(N_j, N_i) = 0$ for all $i < j$. 

Let $T=N\oplus X$ be a basic silting object of $\der^b(\mathcal{T})$. Define 
\[
p_T=\max\{l+1 \mid \Hom_{\der^b(\mathcal{H})}(X,N_i[\mathbb{Z}])=0, 1\leq i\leq l\}.
\]
If $p_T=n$, then $N\in (\thick X)^\perp$ as desired.  If $p_T<n$,  it suffices to  show that there is a silting object $M$ containing $N$ as a direct summand such that $p_M>p_T$.

Now assume that $p:=p_T<n$. 
Let 
\begin{equation}\label{tri:left-mutation}
    X\xrightarrow{f} N'\to Y\to X[1]
\end{equation} 
be the triangle associated with the left mutation of $T$ at $X$. By the definition of $p$, we have $N'\in \add N_p\oplus\cdots\oplus N_{n-1}$, which implies $\Hom_{\der^b(\mathcal{H})}(N',N_i[\mathbb{Z}])=0$ for any $1\leq i<p$. For any $1\leq i<p$, applying $\Hom_{\der^b(\mathcal{H})}(-,N_i[\mathbb{Z}])$ to the triangle \eqref{tri:left-mutation} yields an exact sequence
\[
0=\Hom_{\der^b(\mathcal{H})}(X[1],N_i[\mathbb{Z}])\to \Hom_{\der^b(\mathcal{H})}(Y,N_i[\mathbb{Z}])\to \Hom_{\der^b(\mathcal{H})}(N',N_i[\mathbb{Z}])=0.
\]
Consequently, $\Hom_{\der^b(\mathcal{H})}(Y,N_i[\mathbb{Z}])=0$ and hence $p_{\mu_X^-(T)} \geq p$.

If $N_p\in \add N'$, then $\Hom_{\der^b(\mathcal{H})}(X,N_p)\neq 0$ and $\Hom_{\der^b(\mathcal{H})}(N_p,Y)\neq 0$. Since $Y\oplus N_p$ is a presilting object,  $\Hom_{\der^b(\mathcal{H})}(Y,N_p)=0$ by Lemma \ref{l:no-cycle}. Moreover, $\Hom_{\der^b(\mathcal{H})}(Y,N_p[>0])=0$.  Note that $N_p\in \mathcal{H}[m_p]$ and $\Hom_{\der^b(\mathcal{H})}(N_p,Y)\neq 0$, it follows that $Y\in \mathcal{H}[m_p]$ or $\mathcal{H}[m_p+1]$. We conclude that $\Hom_{\der^b(\mathcal{H})}(Y,N_p[<0])=0$, and hence $\Hom_{\der^b(\mathcal{H})}(Y,N_p[\mathbb{Z}])=0$. Therefore $p_{\mu_X^-(T)}\geq p+1>p_T$ and we may take $M=\mu_X^-(T)$.

Now assume that $N_p\not\in \add N'$. Without loss of generality, we assume that $N'\in \add N_k\oplus\cdots\oplus N_{n-1}$ and $N_k\in \add N'$ for some $p<k\leq n-1$. It follows that $\Hom_{\der^b(\mathcal{H})}(X,N_p)=0$ and $\Hom_{\der^b(\mathcal{H})}(X,N_k)\neq 0$.  By the definition of $p_T$, we know that $\Hom_{\der^b(\mathcal{H})}(X,N_p[-u])\neq 0$ for some $u\geq 1$. Consequently, $X\in \mathcal{H}[m_p-u]$ or $\mathcal{H}[m_p-u-1]$. On the other hand, by $\Hom_{\der^b(\mathcal{H})}(X,N_k)\neq 0$, we have $X\in \mathcal{H}[m_k]$ or $\mathcal{H}[m_k-1]$. Noticing that $m_p\leq \cdots\leq m_k$, we conclude that $m_p=\cdots=m_k$, $u=1$ and $X\in \mathcal{H}[m_p-1]$. Consequently, $Y\in \mathcal{H}[m_p]$ and $\Hom_{\der^b(\mathcal{H})}(Y,N_p)\cong \Hom_{\der^b(\mathcal{H})}(X[1],N_p)\neq 0$. Now consider the left mutation  of $N\oplus Y$ at $Y$. By the above discussion, we conclude that $p_{\mu_Y^-(N\oplus Y)}>p$ and we may take $M=\mu_Y^-(N\oplus Y)$. This completes the proof.
\end{proof}

The following is a generalization of Proposition \ref{p: ind-perp-completion}.
\begin{theorem}\label{t: perp-completion}
Let $N$ be a presilting object in $\der^b(\mathcal{H})$. Then there exists a presilting object $D$, such that $D \oplus N$ is a silting object and $N \in (\thick D)^\perp$. Moreover, $(\thick D, \thick N)$ is a torsion pair of $\der^b(\mathcal{H})$. In particular, $N$ is a silting object in $(\thick D)^\perp$.
\end{theorem}
\begin{proof}
Suppose $N \oplus U$ is a silting object in $\der^b(\mathcal{H})$. Let $U_1$ be an indecomposable direct summand of $U$. By Proposition \ref{p: ind-perp-completion}, there exists a presilting object $D_1$ such that $N \oplus (U/U_1) \oplus D_1$ is silting and $\Hom_{\der^b(\mathcal{H})}(D_1, (N \oplus U/U_1)[\mathbb{Z}]) = 0$. By iterating this procedure for the remaining indecomposable direct summands of $U$, we obtain a silting object $N \oplus D_1 \oplus D_2 \oplus \cdots \oplus D_{|U|}$, which satisfies $\Hom_{\der^b(\mathcal{H})}(D_i, N[\mathbb{Z}]) = 0$ for all $1 \leq i \leq |U|$. Setting $D := \bigoplus_{i=1}^{|U|} D_i$, it follows that $N \in (\operatorname{thick} D)^\perp$.

Following the same argument as the proof of Lemma \ref{l:no-cycle-plus} (b), we have $\thick (N \oplus D) = \thick D * \thick N $. Since $N \oplus D$ is a silting object, we have $\thick (N \oplus D) = \der^b(\mathcal{H})$. This implies that $(\thick D, \thick N)$ is a torsion pair of $\der^b(\mathcal{H})$, and hence $\thick N = (\thick D)^\perp$.
\end{proof}

\section{Reduction of endomorphism algebras}  \label{s:closure properties}
In this section, we investigate the properties of endomorphism algebras of silting objects over hereditary abelian categories, specifically focusing on their behavior under operations related to idempotents and $\tau$-reduction.

Let $\mathcal{H}$ be a hereditary abelian category. Recall that a rigid object $T\in \mathcal{H}$ is called {\em tilting} if $\Hom_{\mathcal{H}}(T, X) = 0 = \Ext_{\mathcal{H}}^1 (T, X)$ implies $X = 0$.
Following \cite[Definition 4.1 and Corollary 4.12]{BZ16b}, a complex $T$ in $\der^b(\mathcal{H})$ is called a \emph{2-term silting complex} over $\mathcal{H}$ if it satisfies the following conditions:
\begin{itemize}
    \item[(S1)] $\Hom_{\der^b(\mathcal{H})}(T, M[i]) = 0$ for every $M$ in $\mathcal{H}$ and all $i \neq 0$ or $1$.
    \item[(S2)] $T$ is a silting object in $\der^b(\mathcal{H})$.
\end{itemize}
A $k$-algebra $A$ is called {\em quasi-tilted} if there exists a hereditary abelian category $\mathcal{H}$ and a basic tilting object $T\in \mathcal{H}$ such that $A \cong\End_{\mathcal{H}}(T)$.
A $k$-algebra $A$ is called {\em quasi-silted} if there exists a hereditary abelian category $\mathcal{H}$ and a basic 2-term silting complex $T$ over $\mathcal{H}$ such that $A \cong \End_{\der^b(\mathcal{H})}(T)$. An Artin algebra $\Lambda$ is called {\em shod} (small homological dimension) if every indecomposable $\Lambda$-module $M$ satisfies either $\pd_\Lambda M \leq 1$ or $\id_\Lambda M \leq 1$. A fundamental property of shod algebras is that their global dimension $ \leq 3$. In particular, a shod algebra $\Lambda$ is called {\em strictly shod} if $\gld \Lambda = 3$. The following result from \cite{BZ16b} establishes the relationship between these two classes of algebras.
\begin{theorem}\label{t: shod and quasisilted}
Let $A$ be a connected finite-dimensional algebra over an algebraically closed field. Then $A$ is shod if and only if it is  quasi-silted.
\end{theorem}

\subsection{Idempotent quotients} For a triangulated category $\mathcal{T}$ and a subcategory $\mathcal{X}$, we denote by $[\mathcal{X}]$ the ideal of $\mathcal{T}$ consisting of morphisms that factor through objects in $\mathcal{X}$.

\begin{lemma}\label{l: ideal of idempotent}
Let $T = D\oplus N$ be a basic silting complex in $\der^b(\mathcal{H})$, and  $e_D$ be the idempotent of $\End_{\der^b(\mathcal{H})}(T)$ corresponding to $D$. Then $\langle e_D \rangle = [\add D](T, T) = [\thick D](T, T)$.
\end{lemma}
\begin{proof}
Note that $\langle e_D \rangle$ consists of those morphisms in $\End_{\der^b(\mathcal{H})}(T)$ that factor through objects in $\add D$. In other words, $\langle e_D \rangle = [\add D](T, T)$. Let $f \in [\thick D](T, T)$, and  assume $f$ factors through $X\in \thick D$ as $T \xra{a} X \xra{b} T$.  By Lemma \ref{lem:thick-subcat-as-extension}, we have
\[
\thick D = \bigcup\limits_{l \geq 0} \add (D[-l]) * \add (D[-l + 1]) * \cdots * \add (D[l]).
\]
Assume $l > 0$ such that $X \in \add (D[-l]) * \add (D[-l + 1]) * \cdots * \add (D[l])$. Then there exists a triangle $X_1 \xra{u} X \to X_2 \to X_1[1]$ with $X_1 \in \add (D[-l]) * \add (D[-l + 1]) * \cdots * \add D$ and $X_2 \in \add (D[1]) * \add (D[2]) * \cdots * \add (D[l])$. Since $T$ is silting, then $\Hom(T, X_2) = 0$ and hence $a = uc$ for some $c : T \to X_1$. There also exists a triangle $X_3 \to X_1 \xra{v} X_4 \to X_3[1]$ with $X_3 \in \add (D[-l]) * \add (D[-l + 1]) * \cdots * \add (D[-1])$ and $X_4 \in \add D$. Again since $T$ is silting, then $\Hom(X_3, T) = 0$ and hence $bu = dv$ for some $d : X_4 \to T$. Then $f = ba = buc = dvc \in [\add D](T, T)$.
\end{proof}

\begin{theorem}\label{p: idempotent-factor-closed} 
Let $A$ be the endomorphism algebra of a basic silting complex over a hereditary abelian category, and $e\in A$ be an idempotent.
\begin{itemize}
    \item[(a)] Then $A / \langle e \rangle$ is isomorphic to the endomorphism algebra of a silting complex over some hereditary abelian category.
    \item[(b)] Moreover, if $A$ is quasi-silted, then $A / \langle e \rangle$ is also quasi-silted.
\end{itemize}
\end{theorem}
\begin{proof}
Suppose $\mathcal{H}$ is a hereditary abelian category and $T\in \der^b(\mathcal{H})$ is a basic silting object with $A\cong \End_{\der^b(\mathcal{H})}(T)$. Let  $T=D\oplus N$ and $e_D$ be the idempotent of $\End_{\der^b(\mathcal{H})}(T)$ corresponding to $D$. It suffices to show that $\End_{\der^b(\mathcal{H})} (T) / \langle e_D \rangle$ is isomorphic to the endomorphism algebra of a basic silting complex over some hereditary abelian category.

Let $\mathcal{S}:=\thick D$. Since $\mathcal{S}$ is functorially finite, we consider the following triangle in $\der^b(\mathcal{H})$
\begin{equation}\label{tri:torsion-pair}
    S\xrightarrow{f} N\xrightarrow{g} S_N\to S[1],
\end{equation}
where $f$ is a minimal right $\mathcal{S}$-approximation of $N$. It follows that $S_N\in \mathcal{S}^\perp$ by Lemma \ref{l:wakamatsu-lemma}.  On the other hand, by Lemma \ref{lem:thick-subcat-as-extension}, one observes that $S\in \mathcal{S}_D^{\leq 0}:=\bigcup_{l\geq 0}\add D\ast\cdots\ast\add D[l]$.

Applying $\Hom_{\der^b(\mathcal{H})}(N,-)$ to the triangle \eqref{tri:torsion-pair} yields an exact sequence
\[
\Hom_{\der^b(\mathcal{H})}(N,S)\xrightarrow{f_*}\Hom_{\der^b(\mathcal{H})}(N,N)\xrightarrow{g_*}\Hom_{\der^b(\mathcal{H})}(N,S_N)\to \Hom_{\der^b(\mathcal{H})}(N,S[1]).
\]
Since $S\in \mathcal{S}_D^{\leq 0}$, $\Hom_{\der^b(\mathcal{H})}(N,S[1])=0$. By Lemma \ref{l: ideal of idempotent}, we obtain \[\Hom_{\der^b(\mathcal{H})}(N,S_N)\cong \Hom_{\der^b(\mathcal{H})}(N,N)/\im f_*=\End_{\der^b(\mathcal{H})}(T)/\langle e_D\rangle.\]
On the other hand, applying $\Hom_{\der^b(\mathcal{H})}(-,S_N)$ to the triangle \eqref{tri:torsion-pair}, we obtain the following exact sequence
\[
\Hom_{\der^b(\mathcal{H})}(S[1],S_N)\to \Hom_{\der^b(\mathcal{H})}(S_N,S_N)\xrightarrow{g^\ast}\Hom_{\der^b(\mathcal{H})}(N,S_N)\to \Hom_{\der^b(\mathcal{H})}(S,S_N),
\]
where $\Hom_{\der^b(\mathcal{H})}(S[1],S_N)=0=\Hom_{\der^b(\mathcal{H})}(S,S_N)$ since $S_N\in \mathcal{S}^{\perp}$.
Consequently, we obtain isomorphisms of vector spaces 
\[
\Hom_{\der^b(\mathcal{H})}(S_N,S_N)\xrightarrow{g^\ast}\Hom_{\der^b(\mathcal{H})}(N,S_N)\xleftarrow{\bar{g}_*}\Hom_{\der^b(\mathcal{H})}(N,N)/\im f_*.
\]
It is routine to check that $\bar{g}_*^{-1}\circ g^\ast$ is an isomorphism of algebras. 

Let $\mathbb{L}:\der^b(\mathcal{H})\to \der^b(\mathcal{H})/\mathcal{S}$ be the localization functor. According to Proposition \ref{p: quotient-by-presilting}, we have $\der^b(\mathcal{H})/\mathcal{S}\cong \der^b(D^{\perp_\mathbb{Z}})$.  Clearly, $\mathbb{L}(T)=\mathbb{L}(S_N)$. By Theorem \ref{t:silting-reduction}, $\mathbb{L}(S_N)$ is a basic silting object of $\der^b(\mathcal{H})/\mathcal{S}$. 
Since $S_N\in \mathcal{S}^\perp$, we conclude that $\End_{\der^b(\mathcal{H})/\mathcal{S}}(\mathbb{L}(N))\cong \End_{\mathcal{S}^{\perp}}(S_N)\cong \End_{\der^b(\mathcal{H})}(S_N)$ by Lemma \ref{l:perp-quotient}, and hence  $\End_{\der^b(\mathcal{H})/\mathcal{S}}(\mathbb{L}(N))\cong \End_{\der^b(\mathcal{H})} (T) / \langle e_D \rangle$.
This finishes the proof of $(a)$.

 Now assume $T$ is 2-term (\ie, $T$ satisfies (S1)) in $\der^b(\mathcal{H})$.
In view of  Proposition \ref{p:identification-subcategory}, we identify  $\der^b(D^{\perp_\mathbb{Z}})$ with $\mathcal{S}^\perp$ via the embedding $\iota:\der^b(D^{\perp_{\mathbb{Z}}})\hookrightarrow \der^b(\mathcal{H})$. Consequently, it remains to verify that $S_N$ is 2-term in $\der^b(D^{\perp_{\mathbb{Z}}})$ with respect to $D^{\perp_{\mathbb{Z}}}$.

Let $M \in D^{\perp_{\mathbb{Z}}}$ and $i \neq 0, 1$. Applying $\Hom_{\der^b(\mathcal{H})} (-, M[i])$ to the triangle \eqref{tri:torsion-pair}, we obtain an exact sequence
\begin{align*}
0 \overset{M \in D^{\perp_{\mathbb{Z}}}}{=} \Hom_{\der^b(\mathcal{H})} (S, M[i - 1]) &\to \Hom_{\der^b(\mathcal{H})} (S_N, M[i])\\
 &\to \Hom_{\der^b(\mathcal{H})} (N, M[i]) \overset{T~\text{is 2-term}}{=} 0.
\end{align*}
It follows that $\Hom_{\der^b(D^{\perp_\mathbb{Z}})} (S_N, M[i])=\Hom_{\der^b(\mathcal{H})} (S_N, M[i])=0$ for $i\neq 0,1$. Hence $S_N$ is a 2-term silting complex over $D^{\perp_\mathbb{Z}}$.
\end{proof}

\begin{remark} \label{r: shod is idem-quotient closed} Let $A$ be a finite-dimensional algebra over an algebraically closed field, and $e\in A$ be an idempotent. By Theorem \ref{t: shod and quasisilted} and Theorem \ref{p: idempotent-factor-closed}  (b), if $A$ is shod, then $A / \langle e \rangle$ is also shod.
\end{remark}

\subsection{Idempotent subalgebras} 

\begin{theorem}\label{p: idempotent-subalg-closed}
Let $A$ be the endomorphism algebra of a basic silting complex over a hereditary category, and $e\in A$ be an idempotent.
\begin{itemize}
    \item[(a)] The subalgebra $eAe$ is isomorphic to the endomorphism algebra of a basic silting complex over some hereditary category.
    \item[(b)] If $A$ is quasi-silted, then $eAe$ is also quasi-silted.
    \item[(c)] If $A$ is quasi-tilted, then $eAe$ is also quasi-tilted.
\end{itemize}
\end{theorem}
\begin{proof}
(a) Let $A = \End_{\der^b(\mathcal{H})}(T)$, where $\mathcal{H}$ is a hereditary abelian category, and $T$ is a basic silting complex in $\der^b(\mathcal{H})$. Assume $T = U\oplus N$, and let $e=e_N$ be the idempotent corresponding to $N$.
We  first observe that $e_N(\End_{\der^b(\mathcal{H})} (T))e_N \cong \End_{\der^b(\mathcal{H})}(N)$. By Theorem \ref{t: perp-completion}, there exists a presilting complex $D$, such that $D \oplus N$ is a basic silting object and $N \in (\thick D)^\perp$. In addition, $N$ is a silting object in $(\thick D)^\perp$. According to Proposition \ref{p: perp-for-presilting}, $D^{\perp_{\mathbb{Z}}}$ is a hereditary abelian category. Furthermore, by  Proposition \ref{p:identification-subcategory}, we identify $(\thick D)^\perp$ with $\der^b(D^{\perp_{\mathbb{Z}}})$. Then $N$ is a basic silting complex in $\der^b(D^{\perp_{\mathbb{Z}}})$, and $\End_{\der^b(\mathcal{H})}(N) \cong \End_{\der^b(\mathcal{D^{\perp_{\mathbb{Z}}}})}(N)$.

(b) Assume $T$ is 2-term in $\der^b(\mathcal{H})$. Since $N$ is a direct summand of $T$, then $N$ is 2-term in $\der^b(\mathcal{H})$, and hence $N$ is 2-term in $\der^b(D^{\perp_{\mathbb{Z}}})$.

(c) Assume $T$ is tilting in $\mathcal{H}$. Then $N \in D^{\perp_{\mathbb{Z}}}$. By \cite[Proposition 4.8]{BZ16b}, $N$ is tilting in $D^{\perp_{\mathbb{Z}}}$.
\end{proof}

\begin{remark} \label{r: shod is idem-subalg closed} Let $A$ be a finite-dimensional algebra over an algebraically closed field, and $e\in A$ be an idempotent. By Theorem \ref{t: shod and quasisilted} and Theorem \ref{p: idempotent-subalg-closed} (b), if $A$ is shod, then $eAe$ is also shod. 
\end{remark}

\begin{remark}
The statements in parts (b) and (c) of Theorem \ref{p: idempotent-subalg-closed} were established in \cite{AC04}. Our method is different from the one in \cite{AC04}.
\end{remark}

\subsection{$\tau$-tilting reduction } \label{s:tau-reduction-closed} 
Let $A$ be a finite-dimensional $k$-algebra and $\mod A$ the category of finite dimensional right $A$-modules. Recall that $M\in \mod A$ is $\tau$-rigid if $\Hom_A(M,\tau M)=0$, where $\tau$ is the Auslander--Reiten translation functor. A $\tau$-rigid module $M\in \mod A$ is {\em $\tau$-tilting} if $|M|=|A|$. A $\tau$-rigid module $M\in \mod A$ is called a {\em support $\tau$-tilting} $A$-module, if there is an idempotent $e\in A$ such that $M$ is a $\tau$-tilting $A/\langle e\rangle$-module. Denote by $s\tau \text{-}\opname{tilt} A$ the set of isomorphism classes of basic support $\tau$-tilting $A$-modules.

Let $Z\in \mod A$ be a $\tau$-rigid $A$-module, then $^\perp (\tau Z):=\{X\in \mod A\mid \Hom_A(X,\tau Z)=0\}$ is a torsion class of $\mod A$. The {\em Bongartz completion} $U_Z:=P(^\perp(\tau Z))$ of $Z$ is the maximal basic direct sum of $\Ext$-projective objects in $^\perp(\tau Z)$, which is a $\tau$-tilting module. Let $B = \End_A (U_Z)$ and $C = B/ \langle e_Z \rangle$, where $e_Z$ is the idempotent of $B$ corresponding to the projective $B$-module $\Hom_A (U_Z, Z)$. The algebra $C$ is called the {\em $\tau$-tilting reduction} of $A$ with respect to $Z$, see \cite{J}.

 Before stating the main result of this subsection. Let us recall a result from \cite{J}.
  Let $\mathcal{T}$ be a $k$-linear $\Hom$-finite Krull-Schmidt triangulated category with a silting object $T$. Let $A = \End_{\mathcal{T}} (T)$, and denote \[\overline{(-)}=\Hom_\mathcal{T} (T, -) : \mathcal{T} \to \mod A.\] 
Let $D$ be a presilting object in $\mathcal{T}$ such that $D \in \add T * \add T[1]$. Let $f : D' \to T[1]$ be a minimal right $(\add D)$-approximation of $T[1]$. By embedding $f$ into a distinguished triangle, we obtain
\begin{equation}\label{tri:addT-addT[1]}
T  \to X_D \to D' \xra{f} T[1].  
\end{equation}
Then $T_D := X_D \oplus D$ lies in 2$_T$-$\opname{silt} \mathcal{T} := \{ T \in \opname{silt} \mathcal{T} ~|~  T \in \add T * \add (T[1]) \}$ (see \cite[Proposition 4.9]{J}). Jasso \cite{J} has established the following connection between silting reduction and $\tau$-tilting reduction.

\begin{theorem}\cite[Theorem 4.12 (a)]{J} \label{t: tau-tilt red and silt red} Let $D$ be a presilting object in $\mathcal{T}$ contained in $\add T * \add T[1]$. Then we have $\End_A (U_{\overline{D}}) / \langle e_{\overline{D}} \rangle \cong \End_{\mathcal{T}/\mathcal{S}} (T_D)$, where $\mathcal{S} = \thick D$ and $e_{\overline{D}}$ is the idempotent of $\End_A (U_{\overline{D}})$ corresponding to the projective $\End_A (U_{\overline{D}})$-module $\Hom_A (U_{\overline{D}}, \overline{D})$.
\end{theorem}

Now we are in a position to state the main result of this subsection.
\begin{theorem} \label{t: quasisilted is tau-tilting reduction closed}Let $A$ be the endomorphism algebra of a basic silting complex over a hereditary abelian category, and $Z$ be a $\tau$-rigid $A$-module.
\begin{itemize}
    \item[(a)] The $\tau$-tilting reduction of $A$ with respect to $Z$ is  isomorphic to the endomorphism algebra of a basic silting complex over some hereditary abelian category.
    \item[(b)] If $A$ is quasi-silted, then the $\tau$-tilting reduction of $A$ with respect to $Z$ remains a quasi-silted algebra.
\end{itemize}
\end{theorem}
\begin{proof}
Let $\mathcal{H}$ be a hereditary abelian category and $T\in \der^b(\mathcal{H})$ a basic silting object with $A\cong \End_{\der^b{\mathcal{H}}}(T)$.
 According to \cite[Theorem 4.5]{IJY}, there is a bijection \[2_T\text{-}\opname{silt} \der^b(\mathcal{H}) \leftrightarrow s\tau \text{-}\opname{tilt} A\] induced by $\overline{(-)} = \Hom_{\der^b(\mathcal{H})} (T, -)$. Thus, there is a presilting complex $D$ in $\der^b(\mathcal{H})$ contained in $\add T * \add T [1]$ such that $Z = \overline{D}$. By Theorem \ref{t: tau-tilt red and silt red}, we have $\End_A (U_{Z}) / \langle e_Z \rangle \cong \End_{\der^b(\mathcal{H})/\mathcal{S}} (T_D)$, where $\mathcal{S} = \thick D$. According to Proposition \ref{p: quotient-by-presilting}, we have $\der^b(D^{\perp_{\mathbb{Z}}})\cong \der^b(\mathcal{H})/\mathcal{S}$.  On the other hand, by the definition of $T_D$, $T_D\cong T$ in $\der^b(\mathcal{H})/\mathcal{S}$. It follows that $T_D$ is a basic silting object of $\der^b(\mathcal{H})/\mathcal{S}$ by Theorem \ref{t:silting-reduction}. Consequently, $\End_{\der^b(\mathcal{H})/\mathcal{S}}(T_D)$ is isomorphic to the endomorphism algebra of a basic silting complex over $D^{\perp_{\mathbb{Z}}}$. This finishes the proof of $(a)$.

To establish part (b), we further assume that $T$ is a 2-term silting complex in $\der^b(\mathcal{H})$. 
Consider the following triangle in $\der^b(\mathcal{H})$
\[
S\xrightarrow{f} T\to S_T\to S[1],
\]
where $f$ is a minimal right $\mathcal{S}$-approximation of $T$. We have $S_T\in \mathcal{S}^\perp$.
We identify $\mathcal{S}^\perp$ with $\der^b(D^{\perp_\mathbb{Z}})$ by Proposition \ref{p:identification-subcategory}. It follows that 
\[
\End_{\der^b(\mathcal{H})/\mathcal{S}}(T_D)\cong \End_{\der^b(\mathcal{H})/\mathcal{S}}(T)\cong \End_{\der^b(D^{\perp_\mathbb{Z}})}(S_T).
\]
It suffices to show that $S_T$ is a 2-term silting complex in $\der^b(D^{\perp_\mathbb{Z}})$.

Note that $S_T \cong T$ is silting in $\der^b(\mathcal{H})/\mathcal{S}$ by Theorem \ref{t:silting-reduction}. Then $S_T$ is silting in $\der^b(D^{\perp_\mathbb{Z}})$. Let $M\in D^{\perp_\mathbb{Z}}$ and $i\neq 0,1$. 
 Applying $\Hom_{\der^b(\mathcal{H})} (-, M[i])$ to the above triangle, we obtain an exact sequence
\begin{align*}
0 \overset{M \in D^{\perp_{\mathbb{Z}}}}{=} \Hom_{\der^b(\mathcal{H})} (S, M[i - 1]) &\to \Hom_{\der^b(\mathcal{H})} (S_T, M[i])\\
 &\to \Hom_{\der^b(\mathcal{H})} (T, M[i]) \overset{T~\text{is 2-term}}{=} 0.
\end{align*}
It follows that $\Hom_{\der^b(D^{\perp_\mathbb{Z}})} (S_T, M[i])=\Hom_{\der^b(\mathcal{H})} (S_T, M[i])=0$ for $i\neq 0,1$. Hence $S_N$ is a 2-term silting complex over $D^{\perp_\mathbb{Z}}$. This completes the proof.

\end{proof}

\begin{remark} \label{r: shod is tau-tilting reduction closed} Let $A$ be a finite-dimensional algebra over an algebraically closed field, and $Z$ be a $\tau$-rigid $A$-module. By Theorem \ref{t: shod and quasisilted} and Theorem \ref{t: quasisilted is tau-tilting reduction closed} (b), if $A$ is shod, then the $\tau$-tilting reduction $\End_A (U_Z) / \langle e_Z \rangle$ is also a shod algebra.
\end{remark}

\begin{remark} \label{c: tau-tilted from shod} For a shod algebra $A$, and a $\tau$-rigid $A$-module $Z$, the endomorphism algebra $\End_A (U_Z)$ may not  be shod, see Example \ref{e: tau-tilted alg from shod is not shod}.
\end{remark}

\section{Idempotent quotients for several classes of algebras} \label{s: Quotient-closed property by idempotents} This section is devoted to proving the following result.

\begin{theorem}\label{t: quotient-closed}
Let $A$ be a finite-dimensional $k$-algebra and $e\in A$ an idempotent.
\begin{enumerate}
\item If $A$ is a laura algebra, then so is $A / \langle e \rangle$.
\item If $A$ is a right (or left) glued algebra, then so is $A / \langle e \rangle$.
\item If $A$ is a weakly shod algebra, then so is $A / \langle e \rangle$.
\item If $A$ is a shod algebra, then so is $A / \langle e \rangle$.
\end{enumerate}
\end{theorem}
\begin{remark} 
Theorem \ref{t: quotient-closed} (4) also follows from Theorem \ref{p: idempotent-factor-closed} (b).
\end{remark}
\begin{remark} \label{c: tilted and quasitilted is not closed under quotient} 
It is worth noting that while the classes of shod, laura, and weakly shod algebras are closed under taking arbitrary idempotent quotients, the same does not hold for tilted or quasi-tilted algebras.
As evidenced by the example in \cite[Section 2.5]{BMR08}, a tilted algebra can yield a quotient that is strictly shod (i.e., shod but not quasi-tilted).

\end{remark}
Let $A$ be an Artin algebra.
Given indecomposable $A$-modules $M$ and $N$, a {\em path} from $M$ to $N$ is a sequence of morphisms
\[ M = X_0 \xra{f_1} X_1 \xra{f_2} X_2 \to \cdots \xra{f_s} X_s = N
\]
where each $X_i\in \ind A$ and each $f_i$ is a nonzero non-isomorphism. The integer $s$ is the \emph{length} of this path. Following Happel, Reiten, and Smal{\o} \cite{HRS}, we define two full subcategories of $\ind A$
\[ \mathcal{L}_A = \{ X \in \ind A ~|~ \text{for any path from } ~W~ \text{to} ~X, ~\text{we have}~ \pd_A W \leq 1 \},
\]
\[ \mathcal{R}_A = \{ Y \in \ind A ~|~ \text{for any path from } ~Y~ \text{to} ~W, ~\text{we have}~ \id_A W \leq 1 \}.
\]
We recall that
\begin{itemize}
    \item $A$ is called \emph{laura} if $\mathcal{L}_A \cup \mathcal{R}_A$ is cofinite in $\ind A$, that is, there are only finitely many modules in $\ind A - \mathcal{L}_A \cup \mathcal{R}_A$; 
    \item $A$ is called \emph{right glued (resp. left glued)} if $\mathcal{L}_A$ (resp. $\mathcal{R}_A$) is cofinite in $\ind A$ (see \cite[Lemma 3.1]{AC04}); 
    \item $A$ is called \emph{weakly shod} if the length of any path from an indcomposable injective module to an indcomposable projective module is bounded.
\end{itemize}


\begin{proposition}\cite[Proposition 1.3]{AC04} \label{p:weakly shod}
An Artin algebra is weakly shod if and only if there exists an $l \geq 0$ such that any path from an indecomposable module not lying in $\mathcal{L}_A$ to an indecomposable module not lying in $\mathcal{R}_A$ has length at most $l$.
\end{proposition}

Henceforth, we assume that $A$ is a finite-dimensional $k$-algebra.

\begin{lemma}\label{l:canonical-sequence}
Let $e\in A$ be an idempotent. For any right $A$-module $M$, there exists a short exact sequence $0\to MeA\to M \to M/MeA\to 0$, where $MeA \in \Fac eA$ and $M/MeA \in (eA)^\perp$.
\end{lemma}
\begin{proof}
Consider a right $(\add eA)$-approximation of $M$, denoted by $f: Q\to M$. This induces the short exact sequence:
\[0\to \im f \to M\xrightarrow{g} M/\im f\to 0,\] where it is immediate that $\im f \in \Fac eA$. Let $h: eA \to M/\im f$ be any morphism. Since $eA$ is projective, $h$ factors through $g$, say $h = gu$ for some $u: eA \to M$. Furthermore, as $f$ is a right $(\add eA)$-approximation, the morphism $u$ factors through $f$, 
implying $u = fv$ for some $v: eA \to Q$. Hence $h = gu = gfv = 0$. This means $M/\im f \in (eA)^\perp$.

We now verify that $\im f = MeA$. First, observe that the image of a right $(\add eA)$-approximation is independent of the specific choice of $f$. To construct an explicit approximation, let $\{x_1e, \dots, x_de\}$ be a $k$-basis for the vector space $Me$. For each $1 \leq i \leq d$, let $\theta_i: eA \to M$ denote the right $A$-morphism defined by left multiplication: $ea \mapsto x_iea$. By the universal property of direct sums, these maps induce a morphism $\phi: (eA)^d \to M$. Under the canonical isomorphism $Me \cong \operatorname{Hom}_A(eA, M)$, the set $\{\theta_1, \dots, \theta_d\}$ forms a $k$-basis for $\operatorname{Hom}_A(eA, M)$. Consequently, $\phi$ constitutes a right $(\add eA)$-approximation of $M$, and it follows that $\im f = \im \phi$. To show the equality of submodules, note that for any element $(ea_1, \dots, ea_d) \in (eA)^d$, we have:$$ \phi(ea_1, \dots, ea_d) = \sum_{i=1}^d x_iea_i \in MeA, $$ which implies $\im \phi \subseteq MeA$. Conversely, $MeA$ is generated by elements of the form $xea$ where $x \in M$ and $a \in A$. Expanding $xe$ in terms of our chosen basis, we write $xe = \sum_{i=1}^d \lambda_i x_i e$ for some $\lambda_i \in k$. Then:$$ xea = \sum_{i=1}^d x_ie(\lambda_i a) \in \im \phi. $$Therefore, $MeA = \im \phi = \im f$, completing the proof.
\end{proof} 

\begin{remark}\label{rmk:canonical-sequence}
The projective $A$-module $eA$ induces a torsion pair $(\Fac eA, (eA)^\perp)$ of $\mod A$.
The sequence in Lemma \ref{l:canonical-sequence} is the canonical sequence of $M$ with respect to  $(\Fac eA, (eA)^\perp)$.
\end{remark}

\begin{lemma}\label{l:Ext-proj}
Let $M \in (eA)^\perp$. Then M is $\Ext$-projective in $(eA)^\perp$ if and only if there is a projective $A$-module $P$ such that $M \cong P/PeA$.
\end{lemma}
\begin{proof}
Note that $(\Fac eA, (eA)^\perp)$ is a torsion pair in $\mod A$. By Lemma \ref{l:canonical-sequence} and \cite[Proposition VI.1.11 (b)]{ASS}, the assertion holds.
\end{proof} 

\begin{lemma}\label{l: duality-properties}
Let $\Du = \Hom_k(-, k)$ be the standard duality. The following properties hold: 
\begin{enumerate}
\item For any $A$-module $X$, we have $\id_A X = \pd_{A^{op}} \Du (X)$ and $\pd_A X = \id_{A^{op}} \Du (X)$. Consequently, $A$ is shod if and only if its opposite algebra $A^{op}$ is shod.
\item $(A / \langle e \rangle)^{op} \cong A^{op} / \langle e \rangle$.
\item For an $A$-module $X$, $X \in (eA)^\perp$ if and only if $\Du(X) \in (A e)^\perp$. In particular, the restriction of $\Du$ induces a duality $\Du: (eA)^\perp \to (A e)^\perp$.
\item For an $A$-module $X \in (eA)^\perp$, $X$ is $\Ext$-projective (resp. $\Ext$-injective) in $(eA)^\perp$ if and only if $\Du(X)$ is $\Ext$-injective (resp. $\Ext$-projective) in $(Ae)^\perp$.
\end{enumerate}
\end{lemma}
\begin{proof}
(1) and (2) are clear. 

For (3), take an $A$-module $X$.  Then
\begin{align*}
X \in (eA)^\perp &\iff \Hom_A (eA, X) = 0 \\
&\iff \Hom_{A^{op}} (\Du(X), \Du(eA)) = 0 \\
&\iff \Hom_{A^{op}} (A e, \Du(X)) = 0 \\
&\iff \Du(X) \in (A e)^\perp.
\end{align*}

Statement (4) follows from (3) by the duality $\Ext_A^1(X, Y) \cong \Ext_{A^{\text{op}}}^1(\Du(Y), \Du(X)$, which holds for any $X, Y \in \mod A$.
\end{proof} 
The following is well-known.
\begin{lemma}\label{l:cat-of-factor-ring}
Let $A$ be a finite-dimensional $k$-algebra, $I$ an ideal of $A$ and $B = A/I$, then the following properties hold.
\begin{enumerate}
\item The category $\mod B$ is identified with a full subcategory of $\mod A$ consisting of modules annihilated by $I$. In particular, a $B$-module is indecomposable if and only if it is  indecomposable as an $A$-module.
\item $\mod B$ is closed under taking factor modules, submodules, kernels and cokernels in $\mod A$.
\item A morphism $f : X \to Y$ in $\mod B$ is a monomorphism (resp. epimorphism)  if and only if it is a monomorphism (resp. epimorphism) when regarded as a morphism in $\mod A$.
\item A sequence $0 \to X \to Y \to Z \to 0$ of $B$-modules is exact in $\mod B$ if and only if it is exact  in $\mod A$.
\end{enumerate}
\end{lemma}

\begin{proposition}\label{p:proj-dim}
Let $M$ be an indecomposable $A$-module belonging to $(eA)^\perp$. If $\pd_A M \leq 1$, then $\pd_{A / \langle e \rangle} M \leq 1$.
\end{proposition}
\begin{proof}
Conversely, assume $\pd_{A / \langle e \rangle} M \geq 2$, and we prove $\pd_A M \geq 2$. 
Recall that for any $A$-module in the perpendicular category $(eA)^\perp \cong \mod A / \langle e \rangle$, being $\Ext$-projective is equivalent to being a projective $A / \langle e \rangle$-module.
 Let $0\to L \to N \to M\to 0$ be a short exact sequence in $\mod A / \langle e \rangle$, where $N$ is a projective $\mod A / \langle e \rangle$-module. By Lemma \ref{l:Ext-proj}, $N \cong P/PeA$ for some projective $A$-module $P$. According to Lemma \ref{l:cat-of-factor-ring} (4), this sequence remains exact when regarded as a sequence in $\mod A$. By taking the pull-back of the morphisms $L \to N$ and $P \to N$, we construct the following commutative diagram with exact rows and columns in $\mod A$:
\[
\xymatrix{
PeA\ar@{^(->}[d] \ar@{=}[r] & PeA\ar@{^(->}[d] \\
X\ar@{^(->}[r]\ar@{>>}[d] &P\ar@{>>}[r]\ar@{>>}[d] &M\ar@{=}[d] \\
L\ar@{^(->}[r] &N\ar@{>>}[r] &M. \\
}
\]

Since we assumed $\pd_{A / \langle e \rangle} M \geq 2$, the kernel $L$ can not be $\Ext$-projective in $(eA)^\perp$. Observing the first column of the pull-back diagram, where $PeA \in \Fac eA$ and $L \in (eA)^\perp$, we identify this as the canonical exact sequence for $X$.  Hence, by Lemma \ref{l:canonical-sequence}, $L \cong X/XeA$. By Lemma \ref{l:Ext-proj} again, $X$ is non-projective in $\mod A$ since $L$ is not $\Ext$-projective in $(eA)^\perp$.  Hence $\pd_A M \geq 2$. This provides the required contradiction.
\end{proof} 

\begin{proposition}\label{p:inj-dim}
Let $M$ be an indecomposable $A$-module belonging to $(eA)^\perp$. If $\id_A M \leq 1$, then $\id_{A / \langle e \rangle} M \leq 1$.
\end{proposition}
\begin{proof}
By Lemma \ref{l: duality-properties} (1), $\pd_{A^{op}}\Du(M) = \id_A M \leq 1$. By Lemma \ref{l: duality-properties} (3), $\Du(M) \in (A e)^\perp$. By Lemma \ref{l: duality-properties} (1), (2) and Proposition \ref{p:proj-dim}, $\id_{A / \langle e \rangle} M = \pd_{A^{op} / \langle e \rangle} \Du(M) \leq 1$.
\end{proof} 

\begin{proposition}\label{p:L and R}
Let $M$ be an indecomposable $A$-module belonging to $(eA)^\perp$. 
\begin{enumerate}
\item If $M \in \mathcal{L}_A$, then  $M \in \mathcal{L}_{A / \langle e \rangle}$;
\item If $M \in \mathcal{R}_A$, then  $M \in \mathcal{R}_{A / \langle e \rangle}$.
\end{enumerate}
\end{proposition}
\begin{proof}
We restrict our attention to the proof of (1), as the argument for (2) follows by dual considerations. Assume there is a path from an indecomposable module $X$ to $M$ in $\ind A / \langle e \rangle$. Then this is also a path from $X$ to $M$ in $\ind A$. Since $M \in \mathcal{L}_A$, then $\pd_A X \leq 1$. By Proposition \ref{p:proj-dim}, we have $\pd_{A / \langle e \rangle} X \leq 1$. It follows that $M \in \mathcal{L}_{A / \langle e \rangle}$.
\end{proof}

With these preparations in hand, we now turn to the proof of Theorem \ref{t: quotient-closed}.

\begin{proof}[Proof of Theorem \ref{t: quotient-closed}]
(1) By Proposition \ref{p:L and R}, for an indecomposable $A / \langle e \rangle$-module $M$, if $M \notin \mathcal{L}_{A / \langle e \rangle} \cup \mathcal{R}_{A / \langle e \rangle}$, then $M \notin \mathcal{L}_A \cup \mathcal{R}_A$. Hence if $\mathcal{L}_A \cup \mathcal{R}_A$ is cofinite, so is $\mathcal{L}_{A / \langle e \rangle} \cup \mathcal{R}_{A / \langle e \rangle}$.

The proof of (2) is similar to (1).

(3) By Proposition \ref{p:weakly shod}, assume $l \geq 0$ is an integer satisfying that any path from an indecomposable module not lying in $\mathcal{L}_A$ to an indecomposable module not lying in $\mathcal{R}_A$ has length at most $l$. We consider a path $X \to \cdots \to Y$ in $\ind A / \langle e \rangle$, where $X$, $Y$ are indecomposable and $X \notin \mathcal{L}_{A / \langle e \rangle}$, $Y \notin \mathcal{R}_{A / \langle e \rangle}$. By Proposition \ref{p:L and R}, we have $X \notin \mathcal{L}_A$, $Y \notin \mathcal{R}_A$. Since this path is also a path in $\ind A$, its length $\leq l$. By Proposition \ref{p:weakly shod} again, we know that $A / \langle e \rangle$ is also weakly shod.

(4) Given an indecomposable $A$-module $M$, we have $\pd_A M \leq 1$ or $\id_A M \leq 1$ since $A$ is shod. According to Proposition \ref{p:proj-dim} and Proposition \ref{p:inj-dim}, this assertion holds immediately.
\end{proof}

\section{Examples}  \label{s:examples}
We introduce a family of algebras $A(n, k)$ defined for $n \geq 2$ and $2 \leq k \leq n$. Specifically, $A(n, k)$ is the bound quiver algebra $kQ/I$ given by the following quiver:
\[
\xymatrix@-0.7pc{ 
1  &&2  \ar[ll]_{\alpha_1} &&\cdots \ar[ll]_{\alpha_2} &&k \ar[ll]_{\alpha_{k-1}}  &&\cdots \ar[ll]_{\alpha_{k}} &&n - 1 \ar[ll]_{\alpha_{n - 2}} &&n \ar[ll]_{~\alpha_{n - 1}}
}
\]
with 
\[
\alpha_2 \alpha_1 = 0, ~\alpha_3 \alpha_2 = 0, ~\cdots, ~\alpha_{k-1} \alpha_{k-2} = 0.
\]
Here, we adopt the convention for the composition of arrows as established in \cite{ASS}.
Clearly, $A(n, 2)$ is a hereditary algebra of type $\mathbb{A}_n$, and $A(n, n)$ is the quotient of $A(n, 2)$ modulo the square of its Jacobson radical.

Fix $n \geq 3$ and $2 \leq k < n$. For each vertex $i$ of the quiver, we denote the corresponding simple module by $S(i)$, its projective cover by $P(i)$, and its injective envelope by $I(i)$.
The Auslander-Reiten quiver of $A(n, k)$ has the following shape
\[
\xymatrix@-1.5pc@R=10pt{
&&&&&&&&&&*+[F]{\bsm n\\ \vdots \\ k - 1 \esm} \ar[dr] \\
&&&&&&&&&{\bsm n-1\\ \vdots \\ k - 1 \esm} \ar[ur] \ar[dr]  &&{\bsm n\\ \vdots \\ k \esm} \ar[dr] \\
&&&&&&&&\cdots  \ar[ur]  &&{\bsm n-1\\ \vdots \\ k \esm} \ar[ur]  &&*+[F]{\cdots}  \ar[dr] \\
&*+[F]{\bsm 2\\1 \esm} \ar[dr]   &&&&&&{\bsm k \\ k-1 \esm} \ar[ur] \ar[dr]  &&\cdots   &&\cdots  \ar[dr]  &&*+[F]{\bsm n\\n-1 \esm} \ar[dr] \\
*+[F]{\bsm 1 \esm} \ar[ur]   &&{\bsm 2 \esm}  &\cdots  &{\bsm k-2 \esm} \ar[dr]   &&{\bsm k-1 \esm} \ar[ur]  &&{\bsm k \esm} \ar[ur] &&\cdots  &&{\bsm n-1 \esm} \ar[ur]  &&*+[F]{\bsm n \esm} \\
&&&&&*+[F]{\bsm k-1\\k-2 \esm} \ar[ur]
}
\]
We determine the global dimension of $A(n, k)$ by computing the projective dimensions of its simple modules $S(i)$. First, consider the case for $1 \leq i \leq k$. The simple module $S(1)$ coincides with the projective cover $P(1)$, so $\pd_{A(n,k)} S(1) = 0$. 
 For $1 < i \leq k$, the quadratic relations $\alpha_{j+1}\alpha_j = 0$ induce the following short exact sequences:
 \[0 \to S(i - 1) \to P(i) \to S(i) \to 0.\] By induction on $i$, it follows that $\pd_{A(n,k)} S(i) = i-1$ for $1 \leq i \leq k$. 
 
 For $i > k$, we have a short exact sequence \[0 \to P(i - 1) \to P(i) \to S(i) \to 0.\] It follows that $\pd_{A(n,k)} S(i) = 1$ for all $k < i \leq n$. Hence $\gld A(n, k) = k - 1$. (By the same argument, one can check $\gld A(n, n) = n - 1$.)

We take $T = P(1) \oplus P(2) \oplus \cdots \oplus P(k-1) \oplus P(n) \oplus I(k + 1) \oplus \cdots \oplus I(n)$. Clearly, $|T| = n = |A(n, k)|$ and it is easy to see $\pd_{A(n, k)} T \leq 1$. A straightforward computation shows that
\[
\tau I(k + 1) = {\bsm n-1\\ \vdots \\ k \esm}, ~\tau I(k + 2) = {\bsm n-1\\ \vdots \\ k+1 \esm}, \cdots, ~\tau I(n) = {\bsm n-1 \esm}.
\]
Since there is no path in this Auslander-Reiten quiver from an injective module to one of these modules, we have 
\[ \Hom_{A(n, k)}(P(n) \oplus I(k + 1) \oplus \cdots \oplus I(n), \tau T) = 0.\]
On the other hand, comparing the dimension vectors of $P(1), \cdots, P(k-1)$ and $\tau I(k+1), \cdots, \tau I(n)$, it is easy to see
\[ \Hom_{A(n, k)}(P(1) \oplus P(2) \oplus \cdots \oplus P(k-1), \tau T) = 0.\]
Hence $\Hom_{A(n, k)}(T, \tau T) = 0$, and then $T$ is a tilting $A(n, k)$-module. 

We now investigate the structure of the endomorphism algebra $B = \End_{A(n, k)}(T)$. Write $T = T_1 \oplus \cdots \oplus T_n$, where $T_i = P(i)$ for $1 \leq i \leq k-1$, $T_k = P(n)$, and $T_i = I(i)$ for $k+1 \leq i \leq n$. By analyzing the distribution of these summands within the Auslander-Reiten quiver, we have the following facts:
\begin{itemize}
\item For $1 \leq i < j \leq n$, there are no paths from $T_j$ to $T_i$ and hence $\Hom_{A(n, k)} (T_j, T_i) = 0$. 

\item For $1 \leq i \leq n - 1$, there exists a unique path in the AR quiver from $T_i$ to $T_{i+1}$ which is sectional. The composition of irreducible morphisms along this path yields a non-zero morphism $f_i : T_i \to T_{i+1}$. Accordingly, $\dim_k \Hom_{A(n, k)}(T_i, T_{i+1}) = 1$, with the space spanned by $f_i$.

\item For $1 \leq i \leq k-1$, $i+2 \leq j \leq n$, by comparing the dimension vectors of $T_i$ and $T_j$, we have $\Hom_{A(n, k)} (T_i, T_j) = 0$. 

\item For $k \leq i \leq n-2$, $i+2 \leq j \leq n$, the space $\Hom_{A(n, k)} (T_i, T_j)$ is a 1-dimensional $k$-vector space spanned by $f_{ij} = f_{j-1}f_{j-2}\cdots f_i$.
\end{itemize}

Denote by $e_i$ the primitive idempotent of $B$ corresponding to $T_i$.  Based on the homological observations above, the structure of the Jacobson radical $\opname{rad} B$ is characterized as follows:
\[e_{i}(\opname{rad} B )e_{j} = 0,~ \text{ for } j > i,\] 
\[e_{i+1}(\opname{rad} B )e_{i} / e_{i+1}(\opname{rad} B )^2 e_{i} \cong  \opname{span}\{f_i\},~ \text{ for } 1\leq i \leq n-1,\] 
\[e_{j}(\opname{rad} B )e_{i} / e_{j}(\opname{rad} B )^2 e_{i} = 0,~ \text{ for } 1 \leq i < j \leq n, j - i \geq 2. \] 
In addition, $f_2 f_1 = 0$, $f_3 f_2 = 0$, $\cdots$, $f_k f_{k-1} = 0$. 
Consequently, the endomorphism algebra $B = \operatorname{End}_{A(n, k)}(T)$ is defined by the following bound quiver:
\[
\xymatrix@-0.7pc{ 
1  &&2  \ar[ll]_{\alpha_1} &&\cdots \ar[ll]_{\alpha_2} &&k \ar[ll]_{\alpha_{k-1}}  &&\cdots \ar[ll]_{\alpha_{k}} &&n - 1 \ar[ll]_{\alpha_{n - 2}} &&n \ar[ll]_{~\alpha_{n - 1}}
}
\]
with
\[
\alpha_2 \alpha_1 = 0, ~\alpha_3 \alpha_2 = 0, ~\cdots, ~\alpha_{k} \alpha_{k-1} = 0.
\]
This implies $B \cong A(n, k+1)$. We have thus shown the following result.

\begin{proposition}\label{p: iterative tilt}
For $n \geq 2$ and $2 \leq k \leq n$, we have $\gld A(n, k) = k - 1$. If $k < n$, there is a tilting $A(n, k)$-module $T$ such that $\End_{A(n, k)} (T) \cong A(n, k + 1)$. In particular, iteratively taking the endomorphism algebras of tilting modules, starting from a hereditary algebra of type $\mathbb{A}_n$, yields an algebra of global dimension $n - 1$ after $n - 2$ steps. Moreover, $A(n, k)$ is isomorphic to the endomorphism algebra of a tilting complex over a hereditary algebra of type $\mathbb{A}_n$.
\end{proposition}
\begin{proof}
It suffices to note that $\gld A(n, n) = n - 1$ and $A(n, 2)$ is a hereditary algebra of type $\mathbb{A}_n$. For the last assertion, it suffices to observe that $A(n, k)$ is derived equivalent to $A(n,2)$, and the preimage of $T$ in $\der^b(A(n, 2))$ is a tilting complex.
\end{proof}

Let $X_0 \xra{f_1} X_1 \xra{f_2} X_2 \to \cdots \xra{f_s} X_s$ be a path of irreducible morphisms in $\ind A$. If $\tau X_{i+1} = X_{i - 1}$ for some $1 \leq i \leq s - 1$, then $X_i$ is called a \emph{hook} in this path (see \cite{CL99}).

\begin{proposition}\label{p: classification of A(n, k)}
For $n \geq 2$ and $2 \leq k \leq n$, the following statements hold:
\begin{enumerate}
\item $A(n, 2)$ is a hereditary algebra.
\item $A(n, 3)$ is a quasi-tilted algebra.
\item $A(n, 4)$ is a strictly shod algebra.
\item $A(n, k)$ is not shod for $k \geq 5$.
\end{enumerate}
\end{proposition}
\begin{proof}
(1) is clear. For (4), we note that $\gld A(n, k) = k - 1 \geq 4$ for $k \geq 5$. Hence $A(n, k)$ is not shod since the global dimension of a shod algebra is not larger than 3.

Let $k = 3$. Observe that, in the Auslander-Reiten quiver of $A(n, 3)$, there is a sectional path from $I(1)$ to $P(n)$, and this is the only path from an injective $A(n, 3)$-module to a projective $A(n, 3)$-module. By \cite[Theorem II.1.14]{HRS}, $A(n, 3)$ is a quasi-tilted algebra.

Let $k = 4$. There is a path from $I(1)$ to $P(n)$ and a path from $I(2)$ to $P(n)$. These are the only paths from an injective $A(n, 4)$-module to a projective $A(n, 4)$-module. The first has one hook and the second is sectional. Hence $A(n, 4)$ is a strictly shod algebra by \cite[Theorem 2.1 and Proposition 2.4]{CL99}.
\end{proof}

\begin{example}  \label{e: tau-tilted alg from quasitilted is not quasitilted}
Let $A$ be the $k$-algebra given by:
\[
\xymatrix@-0.7pc{ 
1  &2  \ar[l]_{\alpha} &3 \ar[l]_{\beta} &4 \ar[l]_{\gamma}
}
\]
with $\beta \alpha = 0$. By Proposition \ref{p: classification of A(n, k)} (2), $A \cong A(4, 3)$ is quasi-tilted. Its Auslander-Reiten quiver exhibits the following structure:
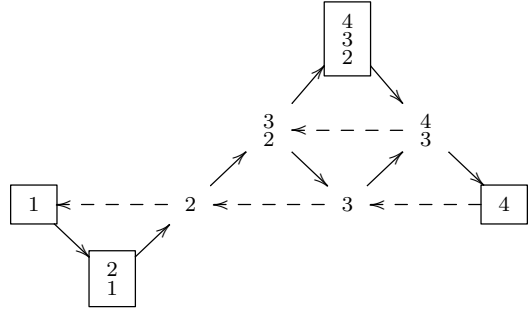
\[
\xymatrix@-1.0pc@R=10pt{
&&&&*+[F]{\bsm 4\\3\\2 \esm} \ar[dr] \\
&&&{\bsm 3\\2 \esm} \ar[ur] \ar[dr]  &&{\bsm 4\\3 \esm} \ar[dr] \ar@{-->}[ll] \\
*+[F]{\bsm 1 \esm} \ar[dr]   &&{\bsm 2 \esm} \ar[ur] \ar@{-->}[ll]  &&{\bsm 3 \esm} \ar[ur] \ar@{-->}[ll]  &&*+[F]{\bsm 4 \esm} \ar@{-->}[ll] \\
&*+[F]{\bsm 2\\1 \esm} \ar[ur]
}
\]
\\
We take $T = P(1) \oplus P(2) \oplus P(4) \oplus I(4)$, which is the Bongartz completion of $I(4)$. The endomorphism algebra $B = \End_A(T)$ is given by:
\[
\xymatrix@-0.7pc{ 
1  &2  \ar[l]_{\alpha} &3 \ar[l]_{\beta} &4 \ar[l]_{\gamma}
}
\]
with $\beta \alpha = 0$, $\gamma \beta = 0$. By Proposition \ref{p: classification of A(n, k)} (3), $B \cong A(4, 4)$ is strictly shod. In particular, it is not quasi-tilted again.
\end{example}

\begin{example}  \label{e: tau-tilted alg from shod is not shod}
Let $A$ be the $k$-algebra given by:
\[
\xymatrix@-0.7pc{ 
1  &2  \ar[l]_{\alpha} &3 \ar[l]_{\beta} &4 \ar[l]_{\gamma} &5 \ar[l]_{\delta}
}
\]
with $\beta \alpha = 0$, $\gamma \beta = 0$. By Proposition \ref{p: classification of A(n, k)} (3), $A \cong A(5, 4)$ is strictly shod. Its Auslander-Reiten quiver is
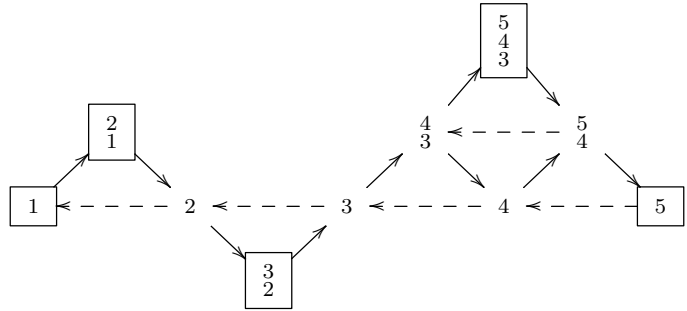
\[
\xymatrix@-1.0pc@R=10pt{
&&&&&&*+[F]{\bsm 5\\4\\3 \esm} \ar[dr] \\
&*+[F]{\bsm 2\\1 \esm}  \ar[dr]   &&&&{\bsm 4\\3 \esm} \ar[ur] \ar[dr] &&{\bsm 5\\4 \esm} \ar[dr] \ar@{-->}[ll] \\
*+[F]{\bsm 1 \esm} \ar[ur]  &&{\bsm 2 \esm} \ar[dr] \ar@{-->}[ll] &&{\bsm 3 \esm} \ar[ur] \ar@{-->}[ll] &&{\bsm 4 \esm}  \ar[ur] \ar@{-->}[ll] && *+[F]{\bsm 5 \esm} \ar@{-->}[ll]\\
&&&*+[F]{\bsm 3\\2 \esm}  \ar[ur]
}
\]
\\
We take $T = P(1) \oplus P(2) \oplus P(3) \oplus P(5) \oplus I(5)$, which is the Bongartz completion of $I(5)$. The endomorphism algebra $B = \End_A(T)$ is given by:
\[
\xymatrix@-0.7pc{ 
1  &2  \ar[l]_{\alpha} &3 \ar[l]_{\beta} &4 \ar[l]_{\gamma} &5 \ar[l]_{\delta}
}
\]
with $\beta \alpha = 0$, $\gamma \beta = 0$, $\delta \gamma = 0$. It is easy to check $\pd_B S(3) = \id_B S(3) = 2$, and hence $B = \End_A(T)$ is not shod.
\end{example}


\def\cprime{$'$} \def\cprime{$'$}
\providecommand{\bysame}{\leavevmode\hbox to3em{\hrulefill}\thinspace}
\providecommand{\MR}{\relax\ifhmode\unskip\space\fi MR }
\providecommand{\MRhref}[2]{%
  \href{http://www.ams.org/mathscinet-getitem?mr=#1}{#2}
}

\end{document}